\documentclass{amsart}

\usepackage{datetime}
\usepackage[all,dvips]{xy}
\usepackage[dvips]{graphicx}

\usepackage[usenames,dvipsnames]{xcolor}
\usepackage{multirow}
\usepackage[all,dvips]{xy}
\usepackage[dvips]{graphicx}
\usepackage[T1]{fontenc}
\usepackage{longtable}
\usepackage{amssymb,amsmath}
\usepackage{pdflscape}
\usepackage{array}

\DeclareMathOperator{\tors}{tors}
\DeclareMathOperator{\coker}{coker}
\DeclareMathOperator{\imagen}{Im} 

\begin{document}

\newtheorem{theorem}{Theorem}
\newtheorem{corollary}[theorem]{Corollary}
\newtheorem{conjecture}[theorem]{Conjecture}
\newtheorem{lemma}[theorem]{Lemma}
\newtheorem{proposition}[theorem]{Proposition}
\newtheorem{axiom}{Axiom}[section]
\newtheorem{exercise}{Exercise}[section]
\newtheorem{problem}{Problem}[section]
\newtheorem{definition}[theorem]{Definition}

\newcommand{\cC}{{\mathcal{C}}}
\newcommand{\cO}{{\mathcal{O}}}
\newcommand{\Q}{{\mathbb{Q}}}
\newcommand{\Z}{{\mathbb{Z}}}
\newcommand{\bP}{{\mathbb{P}}}

\title[Torsion of rational elliptic curves over quadratic fields]{Torsion of rational elliptic curves\\ over quadratic fields}

\author{Enrique Gonz\'alez--Jim\'enez}
\address{Universidad Aut{\'o}noma de Madrid, Departamento de Matem{\'a}ticas and Instituto de Ciencias Matem{\'a}ticas (ICMat), Madrid, Spain}
\email{enrique.gonzalez.jimenez@uam.es}
\urladdr{http://www.uam.es/enrique.gonzalez.jimenez}
\author{Jos\'e M. Tornero}
\address{Departamento de \'Algebra, Universidad de Sevilla. P.O. 1160. 41080 Sevilla, Spain.}
\email{tornero@us.es}
\thanks{The first author was partially  supported by the grant MTM2012--35849. The second author was partially  supported by the grants  FQM--218 and P08--FQM--03894, FSE and FEDER (EU)}


\begin{abstract}
Let $E$ be an elliptic curve defined over $\Q$. We study the relationship between the torsion subgroup $E(\Q)_{\tors}$ and the torsion subgroup $E(K)_{\tors}$, where $K$ is a quadratic number field. 
\end{abstract}
\maketitle

\section{Introduction}

Let $E$ be an elliptic curve defined over a number field $K$. The Mordell-Weil Theorem states that the set of $K$-rational points, $E(K)$, is a finitely generated abelian group. So it can be written as $E(K) \simeq E(K)_{\tors} \oplus \Z^r$, for some non-negative integer $r$ (rank of $E(K)$) and some finite torsion subgroup $E(K)_{\tors}$. It is well known that there exist two positive integers $n,m$ such that $E(K)_{\tors}$ is isomorphic to $\cC_n \times \cC_m$, where $\cC_n$ is the cyclic group of order $n$ \cite{Silverman}.

This paper focuses on a particular problem concerning the torsion part, which we will explain now. We define sets $S(d)$ and $\Phi(d)$ as follows:
\begin{itemize}
\item $S(d)$ is the set of primes that can appear as the order of a torsion point of an elliptic curve $E$ defined over a number field of degree $d$.
\item $\Phi(d)$ is the set of possible groups that can appear as the torsion subgroup of an elliptic curve defined over a certain number field $K$ of degree $d$. 
\end{itemize}

Mazur's landmark papers \cite{Mazur1,Mazur2} established that $S(1)=\{2,3,5,7\}$ and
$$
\Phi(1) = \left\{ \cC_n \; | \; n=1,\dots,10,12 \right\} \cup \left\{ \cC_2 \times \cC_{2m} \; | \; m=1,\dots,4 \right\}.
$$

After this, a long series of papers by Kenku, Momose and Kamienny ending in \cite{Kamienny, KenkuMomose} slowly unfolded the quadratic case to finally reach a full description of $S(2)=\{2,3,5,7,11,13\}$ and $\Phi(2)$:
\begin{eqnarray*}
\Phi(2) &=& \left\{ \cC_n \; | \; n=1,\dots,16,18 \right\} \cup \left\{ \cC_2 \times \cC_{2m} \; | \; m=1,\dots,6 \right\}  \cup \\
&& \left\{ \cC_3 \times \cC_{3r} \; | \; r=1,2 \right\} \cup \left\{ \cC_4 \times \cC_4 \right\}.  
\end{eqnarray*}

We do not have, as of today, such a precise description of $\Phi(d)$ for $d \geq 3$ although work by Parent \cite{Parent} has obtained $S(3)$ and Derickx, Kamienny, Stein and Stoll have announced \cite{Derickx} that they have established the sets $S(d)$ for $d=4,5$. A fundamental result here is the celebrated Uniform Boundness Theorem, a long--standing conjecture finally proved by Merel \cite{Merel}, which states that there exists a constant $B(d)$ such that $|G| \leq B(d)$, for all $G \in \Phi(d)$. Although Merel's proof was not explicit, further versions (an Oesterl\'e's 1994 unpublished paper and Parent \cite{Parent2}) have given precise values for $B(d)$.

\

Let us fix some useful notations:

\begin{itemize}
\item Let $E$ be an elliptic curve defined over a number field. Without loss of generality we can assume $E$ is defined by a short Weierstrass form
$$
E: Y^2 = X^3 + AX + B; \quad A,B \in K,
$$
and we will write,
$$
E(K) = \left\{ (x,y) \in K^2 \; | \; y^2=x^3+Ax+B \right\}\cup \{\cO\},
$$
the set of $K$--rational points of $E$, and $\cO$ its point at infinity.
\item Let $S_\Q(d)$ be the set of primes that can appear as the order of a torsion point defined over a number field of degree $d$, on an elliptic curve $E$ defined over the rationals.
\item Let $\Phi_\Q(d)$ be the set of possible groups that can appear as the torsion subgroup over a number field of degree $d$, of an elliptic curve $E$ defined over the rationals.
\item For an elliptic curve $E$, let $\Delta_E$ be, as customary, its discriminant.
\item For an elliptic curve $E$ and an integer $n$, let $E[n]$ be the subgroup of all points whose order is a divisor of $n$ (over $\overline{\Q}$), and let $E(K)[n]$ be the set of points in $E[n]$ with coordinates in $K$, for any number field $K$ (including the case $K=\Q$).
\item Under the same conditions, let $\Q(E[n])$ be the extension generated by all the coordinates of points in $E[n]$.
\item For an elliptic curve $E$ defined over the rationals given by a short Weierstrass equation $E:Y^2=X^3+AX+B$, and a square--free integer $D$, let $E_D$ denote its quadratic twist. That is, the elliptic curve with a Weierstrass  equation $E_D:DY^2=X^3+AX+B$.
\end{itemize}

Please mind that, in the sequel, for examples and precise curves we will use the Antwerp--Cremona tables and labels \cite{antwerp,cremonaweb}.

The set $S_{\mathbb Q}(d)$ is known for $d \le 42$, and there is a conjectural description by Lozano--Robledo \cite{Lozano} which encompasses both the known cases and experimental data. For example, $S_{\mathbb Q}(1)=S_{\mathbb Q}(2)=\{2,3,5,7\}$. The sets $\Phi_{\mathbb Q}(d)$ have been completely described by Najman \cite{Najman} for $d = 2, 3$. Concretely, he proved:
\begin{eqnarray*}
\Phi_{\mathbb Q}(2) &=& \left\{ \cC_n \; | \; n=1,\dots,10,12,15,16 \right\} \cup \left\{ \cC_2 \times \cC_{2m} \; | \; m=1,\dots,6 \right\} \cup \\
&& \left\{ \cC_3 \times \cC_{3r} \; | \; r=1,2 \right\} \cup \left\{ \cC_4 \times \cC_4 \right\},  \\
\Phi_{\mathbb Q}(3) &=& \left\{ \cC_n \; | \; n=1,\dots,10,12,13,14,18,21 \right\} \cup \left\{ \cC_2 \times \cC_{2m} \; | \; m=1\dots,4,7 \right\}.
\end{eqnarray*}

The remarkable fact is that although the set $\Phi(3)$ is still unknown, Najman has determined the set $\Phi_{\mathbb Q}(3)$. 

It is also worth noting that Fujita \cite{Fujita} has explicitely determined the torsion subgroups over the maximal elementary $2$--extension of $\Q$ (that is $\Q \left(\{\sqrt{m} \; | \; m \in \Z \}\right)$ that may arise from an elliptic curve defined over $\Q$. This classification might be of great help in the open problems that we will pose later.

\begin{definition}
Let $G \in \Phi(1)$. We will write $\Phi_\Q(d,G)$ the set of possible groups that can appear as the torsion subgroup over a certain number field $K$ of degree $d$, of an elliptic curve $E$ defined over the rationals, such that $E(\Q)_{\tors}=G$.
\end{definition}

Our aim in this paper is, first to compute and then to understand better, $\Phi_\Q(2,G)$. That is, the behaviour of a particular torsion group of $\Phi(1)$ when we enlarge the base field $\Q$ by means of a quadratic extension.

In order to guess what $\Phi_\Q(2,G)$ may look like, we carried out an exhaustive computation, obtaining the groups $E(K)_{\tors}$, for all curves $E$ defined over the rationals with conductor less than $300000 $ from \cite{cremonaweb}. The main result of this paper is the following:

\begin{theorem}\label{main}
For $G \in \Phi(1)$, the set $\Phi_\Q(2,G)$ is the following:
$$
\begin{array}{|c|c|}
\hline
G & \Phi_\Q \left(2,G \right)\\
\hline
\cC_1 & \left\{ \cC_1\,,\, \cC_3 \,,\, \cC_5\,,\, \cC_7\,,\, \cC_9 \right\} \\
\hline
\cC_2 & \left\{ \cC_2\,,\, \cC_4 \,,\, \cC_6\,,\, \cC_8\,,\, \cC_{10}\,,\, \cC_{12}\,,\, \cC_{16}\,,\, \cC_2 \times \cC_{2}\,,\, \cC_2 \times \cC_{6}\,,\, \cC_2 \times \cC_{10} \right\} \\ 
\hline
\cC_3 & \left\{ \cC_3, \; \cC_{15}, \; \cC_3 \times \cC_3 \right\} \\
\hline
\cC_4 & \left\{ \cC_4 \,,\,\cC_8\,,\, \cC_{12}\,,\, \cC_2 \times \cC_{4}\,,\, \cC_2 \times \cC_{8}\,,\, \cC_2 \times \cC_{12}\,,\, \cC_4 \times \cC_4\right\} \\
\hline
\cC_5 & \left\{ \cC_5, \; \cC_{15} \right\} \\
\hline
\cC_6& \left\{ \cC_6, \; \cC_{12}, \; \cC_2 \times \cC_6, \; \cC_3 \times \cC_6 \right\} \\
\hline
\cC_7 & \left\{ \cC_7 \right\} \\
\hline
\cC_8 & \left\{ \cC_8, \; \cC_{16}, \; \cC_2 \times \cC_8 \right\} \\
\hline
\cC_9 & \left\{ \cC_9 \right\} \\
\hline
\cC_{10} & \left\{ \cC_{10}, \; \cC_2 \times \cC_{10} \right\} \\
\hline
\cC_{12} & \left\{ \cC_{12}, \; \cC_2 \times \cC_{12} \right\} \\
\hline
\cC_2 \times \cC_2 & \left\{ \cC_2 \times \cC_{2}\,,\, \cC_2 \times \cC_{4}\,,\, \cC_2 \times \cC_{6}\,,\, \cC_2 \times \cC_{8}\,,\,  \cC_2 \times \cC_{12}\right\} \\
\hline
\cC_2 \times \cC_4 & \left\{ \cC_2 \times \cC_4, \; \cC_2 \times \cC_8, \; \cC_4 \times \cC_4 \right\} \\
\hline
\cC_2 \times \cC_6 & \left\{ \cC_2 \times \cC_6, \cC_2 \times \cC_{12} \right\} \\
\hline
\cC_2 \times \cC_8 & \left\{ \cC_2 \times \cC_8 \right\} \\
\hline
\end{array}
$$
\end{theorem}

\section{The set $\Phi_\Q(2,G)$}

A very important and partial result of our problem, concretely the case of $G$ being non--cyclic, has already been completely solved by Kwon \cite{Kwon}. More precisely, the result goes as follows:
\begin{equation}\label{eqkwon}
\Phi_\Q(2,\cC_2 \times \cC_{2n})=
\left\{
\begin{array}{lcl}
\left\{ \cC_2 \times \cC_{2m} \; | \; m=1,2,3,4,6 \right\} & & \mbox{if $n=1$},\\
 \left\{ \cC_2 \times \cC_{2m} \; | \; m=2,4 \right\} \cup \{\cC_4 \times \cC_{4}\}& & \mbox{if $n=2$},\\
 \left\{ \cC_2 \times \cC_{2m} \; | \; m=3,6 \right\} & & \mbox{if $n=3$},\\
 \left\{ \cC_2 \times \cC_{8}  \right\}  & & \mbox{if $n=4$}.
\end{array}
\right.
\end{equation}

The last three cases were carefully detailed, including necessary and sufficient conditions on the coefficients of a short Weierstrass equation of the elliptic curve and the discriminant of the quadratic field, to determine which of the possible groups actually happened for a given case. The first case, however, lacked of such characterisation and it is determined by ruling the other possibilities out. This appears to be a common fact in the study of torsion subgroups: the simpler the structure is, the more difficult becomes to describe in effective terms (see, for instance, \cite{GGT}).

The next auxiliary result is a particular case of Lemma 1.1 in \cite{LaskaLorenz}. However, we need some details from the proof (see \cite{Chahal}).

\begin{theorem}\label{teo3}
Let $E$ be an elliptic curve defined over $\Q$, $D$ a square-free integer and $K = \Q ( \sqrt{D} )$. There exists a pair of homomorphisms
$$
E(K) \stackrel{\Psi}{\longrightarrow} E(\Q) \times E_D(\Q) \stackrel{\overline{\Psi}}{\longrightarrow} E(K)
$$
such that $\overline{\Psi}\circ \Psi=[2]$ and $\Psi\circ \overline{\Psi}=[2]\times [2]$. 

Moreover, $\ker (\Psi) \subset E(K)[2]$, $\ker (\overline{\Psi}) \subset E(\Q)[2] \times E_D(\Q)[2]$ and $\coker(\Psi)$, $\coker(\overline{\Psi})$ are groups where every non--zero element has order $2$.
\end{theorem}

\begin{proof}
Let $\sigma$ be the non--trivial element of $\mbox{Gal}(K/\Q)$ and $\alpha \in K$. Let us write, for any $P = (x,y)$,
$$
\sigma P = ( \sigma(x), \sigma(y) )\quad\mbox{and}\quad\phi(P,\alpha) = (x,\alpha y).
$$
Recall the canonical $K$--isomorphisms between $E$ and its $D$--twist $E_D$:
$$
\xymatrix{E(K) \ar[rr]^{\phi \left(\, \cdot\,,1/\sqrt{D} \right)} & & E_D(K)\ar[rr]^{\phi \left(\, \cdot\,,\sqrt{D} \right)} & & E(K)}
$$
Let $P \in E(K)$. As $\sigma (P + \sigma P) = P + \sigma P$, we have $P + \sigma P \in E(\Q)$. Consider now $P - \sigma P$. We have that, either $\sigma (P - \sigma P) = \cO,$ in which case $P = \sigma P$ and $P \in E(\Q)$, or $\sigma (P - \sigma P) = -(P - \sigma P) \neq \cO.$ Should this be the case, if we write $R = P - \sigma P = (x_R,y_R)$ we have $\sigma (x_R,y_R) = (x_R,-y_R)$. Therefore, $y_R = z_R \sqrt{D}$ for some $z_R \in \Q$, and in this case it is clear that
$$
\phi \left( (x_R,y_R),1 / \sqrt{D} \right) = (x_R, z_R) \in E_D (\Q).
$$
So we define $\Psi$ as follows:
\begin{eqnarray*}
\Psi: E(K) & \longrightarrow & E(\Q) \times E_D(\Q) \\
P & \longmapsto & \left( P + \sigma P, \; \phi \left( P - \sigma P ,1/\sqrt{D} \right) \right)
\end{eqnarray*}
$\Psi$ is clearly a homomorphism: it is straightforward if we write it
$$
\Psi: E(K) \longrightarrow E(K) \times E_D(K),
$$
and we have just shown $\imagen(\Psi) \subset E(\Q) \times E_D(\Q)$.

On the other hand, let us consider a point $R = (x_R,y_R) \in E_D(\Q)$. Clearly 
$$
\phi(R,\sqrt{D}) = \left( x_R,y_R\sqrt{D}  \right) \in E(K),
$$
and we define then
\begin{eqnarray*}
\overline{\Psi}: E(\Q) \times E_D(\Q) & \longrightarrow & E(K) \\
(P,R) & \longmapsto & P + \phi(R,\sqrt{D}).
\end{eqnarray*}
Analogously, it is not difficult to see $\overline{\Psi}$ is a homomorphism. We could have defined (exactly the same way) $\overline{\Psi}: E(K) \times E_D(K) \longrightarrow E(K)$; which is clearly a homomorphism, and then we have it restricted to the subgroup $E(\Q) \times E_D(\Q)$. 

With these definitions, the results from the theorem can be easily deduced.
\end{proof}

\begin{corollary}\label{cor4}
Let $E$ be an elliptic curve defined over $\Q$, $D$ an square-free integer and $K = \Q ( \sqrt{D} )$. If $n$ is odd, then there exists an isomorphism
$$
E(K)[n] \simeq E(\Q)[n] \times E_D(\Q)[n].
$$
\end{corollary}

Our description of $\Phi_\Q(2,G)$, for the eleven cyclic groups in $\Phi(1)$, will rest in the following result:

\begin{theorem}\label{teo}
Let $E$ be an elliptic curve defined over $\Q$, $K$ a quadratic number field, $G\in \Phi_\Q(1)$ and $H\in \Phi_\Q(2)$ such that $E(\Q)_{\tors}\simeq G$ and $E(K)_{\tors}\simeq H$.
\begin{itemize}
\item[(i)] If $\Q(E[3]) = \Q \left( \sqrt{-3} \right)$, then $\cC_3 \subset G$.
\item[(ii)] If $\cC_2 \not\subset G$, then $\cC_2 \not\subset H$.
\item[(iii)] If $G$ is cyclic, $\cC_2 \subseteq G$ and $\cC_4 \not\subset G$, then $\cC_2 \times \cC_4 \not\subset H$.
\item[(iv)] If $G= \cC_4$, then $H\ne \cC_{16}$.
\item[(v)] If $G= \cC_3$, then $H\ne \cC_{9}$.
\item[(vi)] If $H= \cC_{15}$, then $G= \cC_3$ or $G= \cC_5$.
\end{itemize}
\end{theorem}

\begin{proof}
(i) is Proposition 2.1 from \cite{Palladino}. Note (see Section \ref{sec3}) that if $E(K)$ has full $3$--torsion, then $K=\Q \left( \sqrt{-3}\right)$.

(ii) Direct, as the hypothesis implies the irreducibility of $X^3+AX+B$ over $\Q$ (hence over $K$).

(iii) With no loss of generality, assume $E$ to have the form
$$
E: Y^2 = X(X^2+AX+B).
$$

If $\cC_2 \times \cC_2 \subset E(K)$, it must then be $K = \Q( \sqrt{\Delta_E})$, with $\Delta_E = A^2-4B$, which we will call $D$ from now on. The full set of points with order two is
$$
(0,0), \quad \quad \left( -\frac{1}{2}(A-\sqrt{D}),0 \right), 
\quad \quad \left( -\frac{1}{2}(A+\sqrt{D}),0 \right).
$$

Let us assume there is a point of order $4$. Then one of the previous points is the double of such a point, hence (see \cite{Knapp} Thm. 4.2.) one of the following pairs consists of two squares in $K$:
$$
\left\{ \frac{1}{2}(A-\sqrt{D}), \; \frac{1}{2}(A+\sqrt{D}) \right\}, \quad 
\left\{ \frac{1}{2}(-A+\sqrt{D}), \; \sqrt{D} \right\}, \quad 
\left\{ -\sqrt{D}, \; -\frac{1}{2}(A+\sqrt{D}) \right\}.
$$

Only the first possibility can hold. Now an element $a+b\sqrt{D} \in (K^*)^2$ if and only if there exist $x,y \in \Q$ such that
$$
a = x^2 + Dy^2, \quad \quad b=2xy;
$$
so we must have
$$
\left\{ \begin{array}{rcl}
A/2 &=& x^2+Dy^2 \\ 1 &=& 4xy
\end{array} \right. \quad \quad 
\left\{ \begin{array}{rcl}
A/2 &=& t^2+Dz^2 \\ -1 &=& 4tz
\end{array} \right.
$$

Clearly both systems have a solution if and only if one of them has. The first gives rise to
$$
y = \frac{1}{4x}, \quad \quad x^2 = \frac{A}{4} \pm \frac{\sqrt{B}}{2},
$$
so we must assume $B=C^2$ and then $x^2=A/4 \pm C/2$. Therefore, our curve must have the form (renaming $x$ and $C$ as $\alpha/2$ and $\beta$, purely for aesthetic purposes):
$$
E: Y^2 = X(X^2+(\alpha^2+2\beta)X+\beta^2),
$$
but for every $E$ in this familiy of curves we have $\cC_4 \subset E(\Q)$, as the point $(-\beta, \alpha \beta)$ has order $4$. Note that the non--vanishing of $\Delta_E$ is equivalent here to $\alpha^2(\alpha + 4 \beta) \neq 0$.

(iv) Assume we have $K=\Q(\sqrt{D})$ and $E$, with $E(\Q)_{\tors} \simeq \cC_4$, $E(K)\simeq \cC_{16}$. Then, from $\Psi$,
$$
0 \to \ker (\Psi) \stackrel{i}{\longrightarrow} E(K) \stackrel{\Psi}{\longrightarrow} E(\Q) \times E_D(\Q) \stackrel{\pi}{\longrightarrow}
\coker (\Psi) \to 0,
$$
and considering the torsion part, we have $\ker (\Psi)$ is either trivial or $\cC_2$ and $E_D(\Q)_{\tors} \simeq \cC_{2n}$, with $n=1,\dots,6$. So the only possibility is $\ker (\Psi) \simeq \cC_2$, $E_D(\Q)_{\tors} \simeq \cC_8$. But then $\coker (\Psi) \simeq \cC_4$, which contradicts Theorem \ref{teo3}.

(v) Direct from Corollary \ref{cor4}, with $n=9$.

(vi) If $H= \cC_{15}$, then Najman \cite[Theorem 2, c)]{Najman} has shown that $E$ is the elliptic curve  \texttt{50b1} or \texttt{50a3} with $K=\Q (\sqrt{5})$; and  \texttt{50b2} or \texttt{450b4} with $K=\Q( \sqrt{-15})$. Both \texttt{50b1} and \texttt{50b2} have $G=\cC_5$, while \texttt{50a3} and \texttt{450b4} have $G=\cC_3$. Therefore, $H=\cC_{15}$ will only appear in $\Phi_\Q(2,G)$ for $G=\cC_3,\cC_5$.
\end{proof}

\begin{proof}(Theorem \ref{main})
The sets $\Phi_\Q(2,G)$ were first conjectured by the computations mentioned above. The relevant results can be found in Table \ref{table2}. In particular, this shows that all groups said to be in $\Phi_\Q(2,G)$ belong to this set. 

\begin{table}[ht]
\caption{The cases not treated in \cite{Kwon}. The table displays either if the case happens ($\checkmark$), if it is impossible because $G \not\subset H$ ($-$) or if it is ruled out by Theorem \ref{teo} ((i)--(vi)).}\label{table1}
\begin{tabular}{|c|c|c|c|c|c|c|c|c|c|c|c|}
\cline{2-12}
\multicolumn{1}{c|}{} & $\cC_1$ & $\cC_2$ & $\cC_3$ & $\cC_4$ & $\cC_5$ & $\cC_6$ & $\cC_7$ & $\cC_8$ & $\cC_9$ & $\cC_{10}$ & $\cC_{12}$ \\
\hline
$\cC_1$ & $\checkmark$ & $-$ & $-$ & $-$ & $-$ & $-$ & $-$ & $-$ & $-$ & $-$ & $-$ \\
\hline
$\cC_2$ & (ii) & $\checkmark$ & $-$ & $-$ & $-$ & $-$ & $-$ & $-$ & $-$ & $-$ & $-$  \\
\hline
$\cC_3$ & $\checkmark$ & $-$ & $\checkmark$ & $-$ & $-$ & $-$ & $-$ & $-$ & $-$ & $-$ & $-$  \\
\hline
$\cC_4$ & (ii) & $\checkmark$ & $-$ & $\checkmark$ & $-$ & $-$ & $-$ & $-$ & $-$ & $-$ & $-$ \\
\hline
$\cC_5$ & $\checkmark$ & $-$ & $-$ & $-$ & $\checkmark$ & $-$ & $-$ & $-$ & $-$ & $-$ & $-$  \\
\hline
$\cC_6$ & (ii) & $\checkmark$ & (ii) & $-$ & $-$ & $\checkmark$ & $-$ & $-$ & $-$ & $-$ & $-$ \\
\hline
$\cC_7$ & $\checkmark$ & $-$ & $-$ & $-$ & $-$ & $-$ & $\checkmark$ & $-$ & $-$ & $-$ & $-$ \\
\hline
$\cC_8$ & (ii) & $\checkmark$ & $-$ & $\checkmark$ & $-$ & $-$ & $-$ & $\checkmark$ & $-$ & $-$ & $-$  \\
\hline
$\cC_9$ & $\checkmark$ & $-$ & (v) & $-$ & $-$ & $-$ & $-$ & $-$ & $\checkmark$ & $-$ & $-$\\
\hline
$\cC_{10}$ & (ii) & $\checkmark$ & $-$ & $-$ & (ii) & $-$ & $-$ & $-$ & $-$ & $\checkmark$ & $-$  \\
\hline
$\cC_{12}$ & (ii) & $\checkmark$ & (ii) & $\checkmark$ & $-$ & $\checkmark$ & $-$ & $-$ & $-$ & $-$ & $\checkmark$ \\
\hline
$\cC_{15}$ & (vi) & $-$ & $\checkmark$ & $-$ & $\checkmark$ & $-$ & $-$ & $-$ & $-$ & $-$ & $-$ \\
\hline
$\cC_{16}$ & (ii) & $\checkmark$ & $-$ & (iv) & $-$ & $-$ & $-$ & $\checkmark$ & $-$ & $-$ & $-$  \\
\hline
$\cC_2 \times \cC_2$ & (ii) & $\checkmark$ & $-$ & $-$ & $-$ & $-$ & $-$ & $-$ & $-$ & $-$ & $-$  \\
\hline
$\cC_2 \times \cC_4$ & (ii) & (iii) & $\checkmark$ & $-$ & $-$ & $-$ & $-$ & $-$ & $-$ & $-$ & $-$ \\
\hline
$\cC_2 \times \cC_6$ & (ii) & $\checkmark$ & (ii) & $-$ & $-$ & $\checkmark$ & $-$ & $-$ & $-$ & $-$ & $-$ \\
\hline
$\cC_2 \times \cC_8$ & (ii) & (iii) & $-$ & $\checkmark$ & $-$ & $-$ & $-$ & $\checkmark$ & $-$ & $-$ & $-$ \\
\hline
$\cC_2 \times \cC_{10}$ & (ii) & $\checkmark$ & $-$ & $-$ & (ii) & $-$ & $-$ & $-$ & $-$ & $\checkmark$ & $-$ \\
\hline
$\cC_2 \times \cC_{12}$ & (ii) & (iii) & (ii) & $\checkmark$ & $-$ & (iii) & $-$ & $-$ & $-$ & $-$ & $\checkmark$  \\
\hline
$\cC_3 \times \cC_3$ & (i) & $-$ & $\checkmark$ & $-$ & $-$ & $-$ & $-$ & $-$ & $-$ & $-$ & $-$ \\
\hline
$\cC_3 \times \cC_6$ & (i), (ii) & (i) & (ii) & $-$ & $-$ & $\checkmark$ & $-$ & $-$ & $-$ & $-$ & $-$ \\
\hline
$\cC_4 \times \cC_4$ & (ii) & (iii) & $-$ & $\checkmark$ & $-$ & $-$ & $-$ & $-$ & $-$ & $-$ & $-$ \\
\hline
\end{tabular}

\end{table}

The groups $H$ from $\Phi_\Q(2)$ that do not appear in some $\Phi_\Q(2,G)$, with $G < H$ can be ruled out from $\Phi_\Q(2,G)$ most of the times using the previous theorem.

Table \ref{table1} (row $=H$, column $=G$) deals with the case $G \neq \cC_2 \times \cC_{2m}$. The non--cyclic case being treated in (\ref{eqkwon}) (cf. \cite{Kwon}). 

\end{proof}

\section{Real vs. imaginary quadratic extensions}\label{sec3}

A natural question might be if there is any substantial difference between the real and imaginary quadratic case. Our computations (see Section \ref{appendix}) show it is not so, except for some well--known cases. 

In fact, from the Weil pairing \cite{Silverman} we know that, for a number field $K$, if $\cC_m \times \cC_m \subset E(K)$, then $K$ contains the cyclotomic field generated by the $m$--th roots of unity. In the quadratic case, that implies: 
\begin{itemize}
\item $\cC_3 \times \cC_3$ and $\cC_3 \times \cC_6$ can only appear in the imaginary extension $\Q \left( \sqrt{-3} \right)$.
\item $\cC_4 \times \cC_4$  can only appears in the imaginary extension $\Q \left( \sqrt{-1} \right)$.
\end{itemize}

For all the remaining cases, groups in $\Phi_\Q(2)$ appear in both real and imaginary cases. This is shown in Table 2. Consider all elliptic curves, defined over the rationals with:
\begin{itemize}
\item $E(\Q)_{tors}$ as given in the first column
\item $E(\Q(\sqrt{D}))_{tors}$ as given in the second column, for some $D$ as given in the fourth (resp. sixth) column for the real case (resp. imaginary case).
\end{itemize}
Then 
\begin{itemize}
\item The seventh (penultimate) column is the number of curves with conductor less than $300000$ which meet this situation, for a real extension $\Q \subset K$.
\item The eighth (last) column is is the number of curves with conductor less than $300000$ which meet this situation, for an imaginary extension $\Q \subset K$.
\end{itemize}

Notice that the only cases where "$-$" appears are the ones remarked above.

\section{On the number of quadratic extensions with proper extension of the torsion subgroup}

Consider the following problem, closely related to our original one. Take an elliptic curve, defined over the rationals, and allow the base field to be extended to a quadratic number field. How many cases of proper extension in the corresponding torsion groups can we predict to appear?

To begin with, for a fixed curve $E$, then only a small amount of quadratic extensions will be interesting from the point of view of the torsion subgroup. This is a known result; see for instance, \cite[Corollary 2]{Kwon} for a different approach or \cite[Lemma 3.4]{Jeon} for a similar proof than the one presented here, which we include because it suits our forthcoming arguments.

\begin{theorem}\label{divpol}
Let $E$ be an elliptic curve defined over the rationals. Then, for all but finitely many quadratic extensions $K/\Q$, $E(K)_{\tors} = E(\Q)_{\tors}$.
\end{theorem}

\begin{proof} 
The proof is actually the method used in our computations above. 

The order of a torsion point defined over a quadratic extension must be $m \in {\mathcal T} = \{1,\dots,10,12,15,16\}$. Therefore, for a given curve $E$ defined over $\Q$, one only has to compute the $m$--th division polynomials $\psi_m(X)$ ($m$ odd) or $\psi_m(X)/(2Y)$ ($m$ even) \cite{Silverman}, for all $m \in {\mathcal T}$, and look for irreducible quadratic and linear factors in $\Q[X]$. 

For quadratic factors, we must consider their splitting fields. For linear factors $X-\alpha$, we must take the extension (maybe trivial) $\Q\left(\sqrt{\alpha^3+A\alpha+B}\right)$, where $E:y^2=x^3+Ax+B$. These extensions, which are obviously finitely many, are the only quadratic ones where the torsion subgroup may grow.
\end{proof}

One might in fact give an upper bound for the number of quadratic extensions where $E(K)_{\tors} \neq E(\Q)_{\tors}$, simply considering the number of linear factors that may appear in $\psi_m(X)$ or $\psi_m(X)/(2Y)$, with $m \in {\mathcal T}$ being a prime power.

If we want to be more precise we can use Theorem 2 to reduce the number of the division polynomials that must be taken into account. For instance, if $E(\Q)_{\tors}=\cC_3$, we know that $E(K)_{\tors} \in \left\{ \cC_3, \, \cC_5, \cC_{15}, \cC_3 \times \cC_3 \right\}$, hence both torsion subgroups are different if and only if $E(K)$ has either a point of order $5$ or a non--rational point of order $3$. As $\deg(\psi_3)=4$ and $\deg(\psi_5)=12$ there are, at most, $16$ quadratic extensions where the torsion grows.

\begin{corollary}
For any elliptic curve $E$ defined over the rational field, $G=E(\Q)_{\tors}$. The number of quadratic extensions $K$ verifying $E(K)_{\tors} \neq G$ is bounded by a constant $k_G$ that only depends on $G$, and is given by:

\begin{tabular}{|c|c||c|c||c|c||c|c||c||c|}
\hline
$\quad G \quad$ & $k_G$ & $\quad G \quad$ & $k_G$ & $\quad G \quad$ & $k_G$ & $\quad G \quad$ & $k_G$ & $\quad G \quad$ & $k_G$ \\
\hline
\hline
$\cC_1$ & $80$ & $\cC_4$ & $43$ & $\cC_7$ & $0$ & $\cC_{10}$ & $1$ & $\cC_2 \times \cC_4$ & $38$ \\
\hline
$\cC_2$ & $182$ & $\cC_5$ & $4$ & $\cC_8$ & $128$ & $\cC_{12}$ & $1$ & $\cC_2 \times \cC_6$ & $7$ \\
\hline
$\cC_3$ & $16$ & $\cC_6$ & $12$ & $\cC_9$ & $0$ & $\cC_2 \times \cC_2$  & $42$ & $\cC_2 \times \cC_8$ & $0$ \\
\hline

\end{tabular}
\end{corollary}

The result is by no means accurate. In fact, it is not very complicated to sharpen the bound for odd-order groups, the even-order ones being much less understood (in practical terms). Our experimental data suggests that, in fact, the bound might well be $4$ quadratic extensions for all cases.

\

\noindent {\bf Example.--} The elliptic curve \verb|30a7| has minimal Weierstrass equation:
$$
E: Y^2 + XY + Y = X^3 - 5334X - 150368
$$
and $E(\Q)_{tors} = \cC_2$. Even more:
$$
\begin{array}{|c|c|c|c|c|}
\hline
D & -5 & -3 & -2 & -10 \\
\hline
E ( \Q (\sqrt{D}))_{tors} &\cC_4 &\cC_6 &\cC_4& \cC_2 \times \cC_2 \\
\hline
\end{array}
$$

\

Theorem \ref{divpol} allows us to give a different proof of the following result mentioned by Gouv{\^e}a and Mazur \cite{GoveaMazur}.

\begin{theorem}
Given an elliptic curve $E$, defined over $\Q$, there is a finite amount of quadratic twists $E_D(\Q)$ such that $E_D(\Q)$ has points of order greater than $2$.

The number of such quadratic twists is bounded in terms of $G=E(\Q)_{tors}$ as in the previous result.
\end{theorem}

\begin{proof}
Let us consider $K=\Q (\sqrt{D})$ and the morphism $\omega:E(K) \longrightarrow E_D(\Q)$, given by the composition
$$
E(K) \stackrel{\Psi}{\longrightarrow} E(\Q) \times E_D(\Q) 
\stackrel{\pi_2}{\longrightarrow} E_D(\Q),
$$
where $\Psi$ is the mapping given in Theorem 3. That is, 
$$
\omega(P) = \phi( P - \sigma P ,1/\sqrt{D}).
$$

It is clear that $\ker(\omega) = E(\Q)$ and, as in Theorem \ref{teo3}, if we consider the long exact sequence
$$
0 \to E(\Q) \longrightarrow E(K) \stackrel{\omega}{\longrightarrow} E_D(\Q) \longrightarrow \coker(\omega) \to 0,
$$

One can see that, for all $P \in E_D(\Q)$, $[2]P=\omega(\phi(P,\sqrt{D}))$. That is, any non--zero element of $\coker(\omega)$ has order $2$. Now, looking at the finite part, Theorem \ref{divpol} shows that for almost all $D$, we have $E(\Q)_{\tors} = E(K)_{\tors}$, so $\mbox{Img}(\omega)$ is trivial, and $E_D(\Q)_{\tors} \simeq \coker(\omega)$, which has only elements of order at most $2$.
\end{proof}

However, some questions arise in this context which we have not yet an answer to.

\noindent {\bf Problem 1.--} Let us fix $G \in \Phi(1)$. Consider the set of all elliptic curves $E$, defined over $\Q$, with $G = E(\Q)_{tors}$ and, when $E$ varies in this set, give a sharp bound for the maximal number of quadratic extensions $K/\Q$ with $G \neq E(K)_{\tors}$.

\noindent {\bf Problem 2.--} Is there a precise (and easy) description of which are the possible extensions $K/\Q$ with $E(\Q)_{\tors} \neq E(K)_{\tors}$, ideally in terms of some invariant(s) of the curve? 

\noindent {\bf Problem 3.--} For a given $G \in \Phi(1)$, one can see experimentally that not all subsets of $\Phi_\Q(2,G)$ appear when one considers an arbitrary curve $E$ with $G=E(\Q)_{\tors}$ and all quadratic extensions $K/\Q$. Find a precise description of all the combinations that may occur.

In connection to the last problem, one can prove that there are curves with stable torsion:

\begin{proposition}
Let $G \in \Phi(1)$, with $G \neq \cC_{2n}$. Then there exists an elliptic curve $E$ such that, for all quadratic extensions $K/\Q$,
$$
G = E(\Q)_{\tors} = E(K)_{\tors}.
$$
\end{proposition}

\begin{proof}
Clearly the groups $G=\cC_{2n}$ do not satisfy this property, as there is a quadratic extension where full $2$--torsion is achieved. 

The fact that, for all other groups, there are curves with stable torsion for all quadratic extensions can be checked in Table \ref{table2}.
\end{proof}

Finally, some comments about the other possible strategy mentioned at the beginning of the section. If we fix the field, then a thorough study of possible groups is possible \cite{KamiennyNajman} studying the non--cuspidal points of certain modular curves.

This technique was in fact the main tool in the search (and hunt) for the most unusual group appearing in $\Phi_\Q(2)$, which is $\cC_{15}$. As we mentioned before, it only appears in $4$ very specific cases \cite{Najman}.

\section{Computations}\label{appendix}

\begin{table}[ht]
\caption{(See Section 5 for a precise explanation)}\label{table2}
\begin{tabular}{|c|c|c|c|c|c||c|c|}
\hline
\multirow{5}{*}{$\cC_1$} & $\cC_1$ &   \multicolumn{4}{|c||}{\texttt{11a2} } &  \multicolumn{2}{|c|}{898000} \\
\cline{2-6}\cline{7-8}
& $\cC_3$  & \texttt{50b3} & 5 & \texttt{19a2} &  -3 & 38916 & 72257\\
\cline{2-6}\cline{7-8}
&  $\cC_5$ & \texttt{50a4} & 5 & \texttt{99d1} & -3 & 1581 &  2261\\
\cline{2-6}\cline{7-8}
&  $\cC_7$ & \texttt{338f1} &  13& \texttt{208d1} &  -1 & 229 & 295\\
\cline{2-6}\cline{7-8}
& $\cC_9$ & \texttt{432e3} &  3 & \texttt{54a2}  &  -3 & 87 & 105\\
\hline
\multirow{9}{*}{$\cC_2$} & $\cC_4$ & \texttt{15a6} & 5 & \texttt{15a5} & -1 & 105300 & 119253\\
\cline{2-6}\cline{7-8}
& $\cC_6$ & \texttt{80b3} &  3 & \texttt{14a3} & -3 & 10594 & 15658\\
\cline{2-6}\cline{7-8}
& $\cC_8$  & \texttt{72a5} & 3 & \texttt{24a6} &  -1 & 1026 & 1014\\
\cline{2-6}\cline{7-8}
 & $\cC_{10}$  & \texttt{150b3} & 5 & \texttt{198e1} & -3 & 202 & 234\\
\cline{2-6}\cline{7-8}
 & $\cC_{12} $& \texttt{240b3} & 3 & \texttt{30a3} & -3 & 99 & 119\\
\cline{2-6}\cline{7-8}
 & $\cC_{16}$  & \texttt{22050eo1} & 105 & \texttt{3150bk1} & -15 & 3 & 1\\
\cline{2-6}\cline{7-8}
 & $\cC_2\times\cC_2$ & \texttt{14a5} & 2 & \texttt{14a3} &  -7 & 488583 & 256109\\
\cline{2-6}\cline{7-8}
 & $\cC_2\times\cC_6$ & \texttt{100a1} & 5 & \texttt{36a3} &  -3 & 322 & 257\\
\cline{2-6}\cline{7-8}
 & $\cC_2\times\cC_{10}$ & \texttt{2178m1} & 33 & \texttt{450a3}&  -15 & 6 & 4\\
\hline 
\multirow{3}{*}{$\cC_3$} & $\cC_3$ &   \multicolumn{4}{|c||}{\texttt{19a3} }&  \multicolumn{2}{|c|}{33340} \\
\cline{2-6}\cline{7-8}
 & $\cC_{15}$  & \texttt{50a3} & 5 & \texttt{450b4} & -15 & 1 &  1\\
\cline{2-6}\cline{7-8}
  & $\cC_3\times\cC_3$ & $-$ & $-$ & \texttt{19a1} &  -3 & $-$ & 1710\\
\hline
\multirow{6}{*}{$\cC_4$} & $\cC_8$ & \texttt{15a7} &  3 & \texttt{15a8} & -3 & 2403 &	 1244\\
\cline{2-6}\cline{7-8}
 & $\cC_{12}$ & \texttt{150c1}&  5 & \texttt{90c1} &  -3& 56 & 72\\
\cline{2-6}\cline{7-8}
 & $\cC_2\times \cC_{4}$ & \texttt{15a7} &  15 & \texttt{15a8} & -15 & 13990 & 9271\\
\cline{2-6}\cline{7-8}
 & $\cC_2\times \cC_{8}$ & \texttt{1344m5} & 2 & \texttt{192c6} &  -2 & 11 & 18\\
\cline{2-6}\cline{7-8}
 & $\cC_2\times \cC_{12}$ & \texttt{112710cj1} & 17 & \texttt{150c3} & -15 & 3 & 2\\
\cline{2-6}\cline{7-8}
 & $\cC_4\times \cC_{4}$  & $-$ & $-$ & \texttt{40a4} & -1 & $-$ &  56\\
\hline
 \multirow{2}{*}{$\cC_5$} & $\cC_5$ &   \multicolumn{4}{|c||}{\texttt{11a1} } &   \multicolumn{2}{|c|}{ 1127} \\
\cline{2-6}\cline{7-8}
 & $\cC_{15}$ & \texttt{50b1} &  5 & \texttt{50b2} &  -15 & 1 & 1 \\
 \hline
 \multirow{3}{*}{$\cC_6$} & $\cC_{12}$ & \texttt{30a1} & 5& \texttt{30a1} & -3 & 157 & 167\\
\cline{2-6}\cline{7-8}
 &  $\cC_2\times \cC_{6}$  & \texttt{14a2} & 2 & \texttt{14a1} &  -7 & 3431 & 1652\\
\cline{2-6}\cline{7-8}
 & $\cC_3\times \cC_{6}$  & $-$ & $-$ & \texttt{14a1} & -3 & - & 64\\
 \hline
 $\cC_7$ & $\cC_7$ &   \multicolumn{4}{|c||}{\texttt{26b1} }&   \multicolumn{2}{|c|}{ 66} \\
 \hline
 \multirow{2}{*}{$\cC_8$} & $\cC_{16}$ & \texttt{210e1} & 105 &  \texttt{210e1} & -15 & 12 & 6\\
\cline{2-6}\cline{7-8}
 & $\cC_2\times \cC_{8}$ & \texttt{21a3} & 7 & \texttt{15a4} &  -1 & 85 & 64\\
 \hline
 $\cC_9$ & $\cC_9$ &   \multicolumn{4}{|c||}{\texttt{54b3} } & \multicolumn{2}{|c|}{17} \\
 \hline
 $\cC_{10}$ & $\cC_2\times \cC_{10}$  & \texttt{66c1} &  33 & \texttt{66c2} &  -2 & 25 & 11\\
 \hline
 $\cC_{12}$ & $\cC_2\times \cC_{12}$ & \texttt{2730bd1} & 65 & \texttt{90c3}, & -15 & 10 & 5\\
 \hline
 \multirow{5}{*}{$\cC_2\times \cC_{2}$} & $\cC_2\times \cC_{2}$ &   \multicolumn{4}{|c||}{\texttt{120b2} } & \multicolumn{2}{|c|}{36913} \\
\cline{2-6}\cline{7-8}
 & $\cC_2\times \cC_{4}$  & \texttt{15a2} & 5 & \texttt{15a2} &  -1 & 17911 & 12914\\
\cline{2-6}\cline{7-8}
 & $\cC_2\times \cC_{6}$ & \texttt{150c2} &  5 & \texttt{30a6} & -3 & 370 & 459\\
\cline{2-6}\cline{7-8}
 &  $\cC_2\times \cC_{8}$  & \texttt{72a4} &  3 & \texttt{63a2} &  -3 & 61 & 47\\
\cline{2-6}\cline{7-8}
 & $\cC_2\times \cC_{12}$  & \texttt{960o6} &  6 & \texttt{167310w2} & -39 & 4 & 1\\
\hline
 \multirow{3}{*}{$\cC_2\times \cC_{4}$} & $\cC_2\times \cC_{4}$ &   \multicolumn{4}{|c||}{\texttt{24a1} } & \multicolumn{2}{|c|}{ 1054} \\
 \cline{2-6}\cline{7-8}
 &  $\cC_2\times \cC_{8}$  & \texttt{15a1} &  5 & \texttt{21a1} & -3 & 146 & 61\\
\cline{2-6}\cline{7-8}
 &  $\cC_4\times \cC_{4}$ & $-$ & $-$ & \texttt{15a1} & -1 & $-$ & 64\\
 \hline
 \multirow{2}{*}{$\cC_2\times \cC_{6}$} & $\cC_2\times \cC_{6}$ &   \multicolumn{4}{|c||}{\texttt{30a2} } &   \multicolumn{2}{|c|}{71 } \\
\cline{2-6}\cline{7-8}
& $\cC_2\times \cC_{12}$  & \texttt{90c6} &  6 & \texttt{2730bd2} & -14 & 8 & 7\\
\hline
$\cC_2\times \cC_{8}$ & $\cC_2\times \cC_{8}$ &   \multicolumn{4}{|c||}{\texttt{210e2} } &   \multicolumn{2}{|c|}{6 } \\
\hline
\end{tabular}
\end{table}

Table \ref{table2} is the result of our computations with the curves in the Antwerp--Cremona tables, with conductor less than $300000$ (a total of $1887909$ elliptic curves) must be read as follows:

\begin{enumerate}
\item The first column is $G \in \Phi (1)$.
\item The second column is $H \in \Phi_2(\Q,G)$.
\item Columns $3$rd to $6$th display specific examples: 

\begin{itemize}
\item The third column is the elliptic curve $E$ with minimal conductor such that $E(\Q)_{\tors}\simeq G$, $E(K)_{\tors}\simeq H$, with $K=\Q (\sqrt{D})$, $D>0$ being the integer in the fourth column.
\item Columns $5$th and $6$th are analogous, with $D<0$.
\item When these four colums are merged, $H=G$ and the curve in the cell verifies $E(\Q)_{\tors} = E(K)_{\tors}$ for all quadratic extensions $K/\Q$.
\end{itemize}
\item Columns $7$th and $8$th indicate the total amount of curves found verifying the corresponding situation ($7$th for real extensions, $8$th for complex, merged for stable torsion groups).
\end{enumerate}

\

{\bf Acknowledgement:} The authors would like to thank J. Silverman and N. Elkies for their suggestions and ideas.


\end{document}